\documentclass[a4paper]{amsart}

\usepackage[T1]{fontenc}
\usepackage[utf8]{inputenc}
\usepackage[english]{babel}
\usepackage{microtype}
\usepackage{graphicx}
\usepackage{xcolor}
\usepackage{csquotes}

\usepackage[backend=biber,sorting=nyt,giveninits=true,maxbibnames=9]{biblatex}

\usepackage{xpatch}

\makeatletter
\patchcmd\blx@bblinput{\blx@blxinit}
                      {\blx@blxinit
                      }{}{\fail}
\makeatother

\usepackage{mathtools, amsthm, amssymb}
\usepackage{bm, dsfont}
\numberwithin{equation}{section}

\newtheorem{theorem}{Theorem}[section]
\newtheorem{proposition}[theorem]{Proposition}
\newtheorem{corollary}[theorem]{Corollary}
\newtheorem{lemma}[theorem]{Lemma}
\newtheorem{remark}[theorem]{Remark}

\DeclareMathOperator{\trace}{trace}
\DeclareMathOperator{\Span}{span}
\DeclareMathOperator{\diag}{diag}

\DeclareMathOperator{\argmin}{arg\,min}
\DeclareMathOperator{\Id}{Id}

\newif\ifshowchanges
\showchangestrue
\definecolor{revisionblue}{RGB}{0,70,160}

\begin{document}

\title{Required Number of Points in $L_2$~Marcinkiewicz--Zygmund Inequalities}

\author{Felix Bartel}
\address{
    Mathematical Institute for Machine Learning and Data Science (MIDS), Catholic University of Eichstätt--Ingolstadt,
    Auf der Schanz 49,
    85049 Ingolstadt, Germany
}
\email{felix.bartel@ku.de}

\begin{abstract}
    We determine, up to absolute constants, the worst-case number of point evaluations required for a weighted $L_2$ Marcinkiewicz--Zygmund inequality for an $m$-dimensional complex function space.
    If $0<\varepsilon<1$ is the relative distortion, this number is
    \begin{equation*}
        \Theta\Big(\min\Big\{m^2,\frac{m}{\varepsilon^2}\Big\}\Big) ,
    \end{equation*}
    and exact discretization has the sharp worst-case value $m^2$.
    While the upper bounds follow from recent constructions, our contribution is the construction of function spaces that are hard to discretize and yield matching lower bounds.
    We use a trace-variance inequality for weighted subframes of unit-norm tight frames.
    One such instance is the complete-graph edge frame, which yields a construction in every dimension.
    Singer equiangular tight frames improve the constant when $m-1$ is a prime power, while maximal equiangular tight frames give the strongest bound possible using our method whenever they exist.
    We also derive consequences for the conditioning of weighted least-squares systems and for standard condition-number-based iteration estimates when these systems are solved by LSQR.
\end{abstract}

\subjclass[2020]{
    Primary
    41A17; Secondary
    65Y20, 68Q25 }
\keywords{Marcinkiewicz--Zygmund inequality, sampling complexity, weighted least squares, equiangular tight frame, spectral sparsification}

\maketitle

\section{Introduction}

Marcinkiewicz--Zygmund inequalities relate the continuous $L_2$ norm to finitely many weighted point evaluations.
They originated in the work of Marcinkiewicz and Zygmund \cite{MZ37} and are studied in general and specific settings, see e.g., \cite{Filbir11,Temlyakov18,DTT19,KKLT22}.
They are a central tool in approximation theory and in the analysis of weighted least-squares methods, see e.g., \cite{CDL13,BKPU23,Bartel23,CDKU26}.

Let $(\Omega,\nu)$ be a measure space and let $\mathcal F\subseteq L_2(\Omega,\nu)$ be an $m$-dimensional complex function space on which pointwise function evaluations are well-defined.
Distinct points $x_1,\dots,x_n\in\Omega$ and positive weights $w_1,\dots,w_n>0$ satisfy the weighted \emph{$L_2$ Marcinkiewicz--Zygmund (MZ) inequalities} with distortion $0\le\varepsilon<1$ if
\begin{equation}
    (1-\varepsilon)\|f\|_{L_2}^{2}
    \le \sum_{i=1}^{n} w_i |f(x_i)|^2
    \le (1+\varepsilon)\|f\|_{L_2}^{2}
    \quad\text{for all }f\in \mathcal F .
    \tag{MZ}\label{eq:mz}
\end{equation}
Repeated sampling points do not improve the support size, since their weights can be combined; we therefore only count distinct points.
Adopting the notation from \cite{DTT19}, we will write $\mathcal F\in \mathcal M_+^w(n,\varepsilon)$ whenever points and weights satisfying \eqref{eq:mz} exist.
We measure the worst-case sampling complexity by
\begin{equation*}
    \mathcal N(m,\varepsilon)
    \coloneqq \sup_{(\Omega,\nu), \mathcal F}
    \min\Big\{n\in\mathds N : \mathcal F\in\mathcal M_+^w(n,\varepsilon)\Big\} ,
\end{equation*}
where the supremum is over all $m$-dimensional complex function spaces.
Thus, $\mathcal N(m,\varepsilon)$ is the worst-case information complexity of stable weighted norm discretization by point evaluations.
The emphasis on the supremum is essential, as some structured spaces admit ``perfect'' $L_2$ MZ inequalities with $n=m$ and $\varepsilon=0$.
For instance, trigonometric polynomials sampled on an appropriate equispaced grid provide a basic example.
The question is whether there are function spaces for which the distortion cannot be reduced regardless of the choice of points and weights.

Several upper bounds are known.
Independent random sampling from the Christoffel distribution gives, with high probability, a sufficient sample size $n$ of order $m\log m/\varepsilon^2$, with an explicit bound in \cite{CDL13}.
Kadison--Singer methods yield existence results that remove the logarithm, cf.~\cite{NSU21}, and the constructive frame-subsampling version in \cite{BSU23} yields, for $0<\varepsilon<1$,
\begin{equation}\label{eq:constructive-upper}
    \mathcal N(m,\varepsilon)
    \le \left\lceil m \frac{1+\sqrt{1-\varepsilon^2}}{1-\sqrt{1-\varepsilon^2}}\right\rceil
    \le \frac{5m}{\varepsilon^2} .
\end{equation}
These estimates deteriorate as $\varepsilon\searrow 0$.
Exact discretization results in \cite{BKPSU25, ST25} give the uniform upper bound
\begin{equation}\label{eq:exact-upper}
    \mathcal N(m,\varepsilon) \le m^2.
\end{equation}

Our main result shows that this trade-off between distortion and the required number of points is unavoidable in the worst case.

\begin{theorem}\label{maintheorem}
    For every $m\ge 2$ and $0<\varepsilon<1$, we have
    \begin{equation*}
        \max\Big\{m,\frac{m(m+1)(1-\varepsilon^2)}{2(1+m\varepsilon^2)} \Big\}
        \le \mathcal N(m,\varepsilon)
        \le \min\Big\{m^2,\frac{5m}{\varepsilon^2}\Big\} .
    \end{equation*}
    Moreover, $\mathcal N(m,0)=m^2$.
    Consequently,
    \begin{equation}\label{eq:theta}
        \frac 18 \min\Big\{m^2,\frac{m}{\varepsilon^2}\Big\}
        \le \mathcal N(m,\varepsilon)
        \le 5 \min\Big\{m^2,\frac{m}{\varepsilon^2}\Big\}.
    \end{equation}
\end{theorem}

The phase transition in \eqref{eq:theta} occurs at $\varepsilon\asymp m^{-1/2}$.
For $\varepsilon < m^{-1/2}$ the quadratic exact-discretization scale $m^2$ is necessary in the worst case, while for $\varepsilon > m^{-1/2}$ the lower bound has the $m/\varepsilon^2$ behavior.
The endpoint $\varepsilon=0$ is sharper and uses a separate linear independence argument.

The lower bounds come from hard-to-discretize function spaces based on unit-norm tight frames $\{\bm\varphi_i\}_{i=1}^{N}$.
We center the rank-one projectors
\begin{equation*}
    Q_i = \langle\cdot,\bm\varphi_i\rangle\bm\varphi_i-\frac 1m \Id,
\end{equation*}
and measure the smallest eigenvalue $\gamma_\Phi$ of their Gram matrix on $\bm 1_N^\perp$.
A small weighted subset has an unavoidable variance, while an operator satisfying \eqref{eq:mz} must be close to the scalar matrix.
Combining these two facts produces a lower bound depending on $N$ and $\gamma_\Phi$.

For the edge frame of the complete graph, used in Theorem~\ref{maintheorem}, with $m+1$ vertices we obtain $\gamma_\Phi\ge 1/2$, with equality for $m\ge3$.
In this instance \eqref{eq:mz} is equivalent to the spectral sparsification of the complete graph.
We also show that equiangular tight frames (ETFs) maximize $\gamma_\Phi$ for fixed $m$ and $N$.
Consequently, whenever a complex ETF with $N$ vectors in dimension $m$ exists, we obtain the stronger estimate
\begin{equation}\label{eq:etfbound}
    \mathcal N(m,\varepsilon)
    \ge \max\Big\{m,\frac{N(1-\varepsilon^2)}{1+\frac{N-m}{m-1}\varepsilon^2}\Big\}.
\end{equation}
Singer difference sets yield ETFs with $N=m^2-m+1$ elements whenever $m-1$ is a prime power.
A maximal complex ETF has $N=m^2$ elements.
Conditional on Zauner's conjecture, such frames exist for every dimension and \eqref{eq:etfbound} becomes
\begin{equation*}
    \mathcal N(m,\varepsilon)
    \ge \max\Big\{m,\frac{m^2(1-\varepsilon^2)}{1+m\varepsilon^2}\Big\}.
\end{equation*}
This is the strongest bound obtainable from the trace-variance method developed here.

All bounds are shown in Figure~\ref{fig:eps} for a fixed dimension $m=998$.

\begin{figure}
    \begingroup
  \makeatletter
  \providecommand\color[2][]{\GenericError{(gnuplot) \space\space\space\@spaces}{Package color not loaded in conjunction with
      terminal option `colourtext'}{See the gnuplot documentation for explanation.}{Either use 'blacktext' in gnuplot or load the package
      color.sty in LaTeX.}\renewcommand\color[2][]{}}\providecommand\includegraphics[2][]{\GenericError{(gnuplot) \space\space\space\@spaces}{Package graphicx or graphics not loaded}{See the gnuplot documentation for explanation.}{The gnuplot epslatex terminal needs graphicx.sty or graphics.sty.}\renewcommand\includegraphics[2][]{}}\providecommand\rotatebox[2]{#2}\@ifundefined{ifGPcolor}{\newif\ifGPcolor
    \GPcolortrue
  }{}\@ifundefined{ifGPblacktext}{\newif\ifGPblacktext
    \GPblacktexttrue
  }{}\let\gplgaddtomacro\g@addto@macro
\gdef\gplbacktext{}\gdef\gplfronttext{}\makeatother
  \ifGPblacktext
\def\colorrgb#1{}\def\colorgray#1{}\else
\ifGPcolor
      \def\colorrgb#1{\color[rgb]{#1}}\def\colorgray#1{\color[gray]{#1}}\expandafter\def\csname LTw\endcsname{\color{white}}\expandafter\def\csname LTb\endcsname{\color{black}}\expandafter\def\csname LTa\endcsname{\color{black}}\expandafter\def\csname LT0\endcsname{\color[rgb]{1,0,0}}\expandafter\def\csname LT1\endcsname{\color[rgb]{0,1,0}}\expandafter\def\csname LT2\endcsname{\color[rgb]{0,0,1}}\expandafter\def\csname LT3\endcsname{\color[rgb]{1,0,1}}\expandafter\def\csname LT4\endcsname{\color[rgb]{0,1,1}}\expandafter\def\csname LT5\endcsname{\color[rgb]{1,1,0}}\expandafter\def\csname LT6\endcsname{\color[rgb]{0,0,0}}\expandafter\def\csname LT7\endcsname{\color[rgb]{1,0.3,0}}\expandafter\def\csname LT8\endcsname{\color[rgb]{0.5,0.5,0.5}}\else
\def\colorrgb#1{\color{black}}\def\colorgray#1{\color[gray]{#1}}\expandafter\def\csname LTw\endcsname{\color{white}}\expandafter\def\csname LTb\endcsname{\color{black}}\expandafter\def\csname LTa\endcsname{\color{black}}\expandafter\def\csname LT0\endcsname{\color{black}}\expandafter\def\csname LT1\endcsname{\color{black}}\expandafter\def\csname LT2\endcsname{\color{black}}\expandafter\def\csname LT3\endcsname{\color{black}}\expandafter\def\csname LT4\endcsname{\color{black}}\expandafter\def\csname LT5\endcsname{\color{black}}\expandafter\def\csname LT6\endcsname{\color{black}}\expandafter\def\csname LT7\endcsname{\color{black}}\expandafter\def\csname LT8\endcsname{\color{black}}\fi
  \fi
    \setlength{\unitlength}{0.0500bp}\ifx\gptboxheight\undefined \newlength{\gptboxheight}\newlength{\gptboxwidth}\newsavebox{\gptboxtext}\fi \setlength{\fboxrule}{0.5pt}\setlength{\fboxsep}{1pt}\definecolor{tbcol}{rgb}{1,1,1}\begin{picture}(7080.00,2820.00)\gplgaddtomacro\gplbacktext{\colorrgb{0.00,0.00,0.00}\put(818,307){\makebox(0,0){\strut{}$0$}}\colorrgb{0.00,0.00,0.00}\put(1962,307){\makebox(0,0){\strut{}$0.5$}}\colorrgb{0.00,0.00,0.00}\put(3105,307){\makebox(0,0){\strut{}$1$}}\colorrgb{0.00,0.00,0.00}\put(3273,969){\makebox(0,0)[l]{\strut{}$m$}}\colorrgb{0.00,0.00,0.00}\put(3273,2341){\makebox(0,0)[l]{\strut{}$m^2$}}}\gplgaddtomacro\gplfronttext{\csname LTb\endcsname \put(4259,2646){\makebox(0,0)[l]{\strut{}exact discretization}}\csname LTb\endcsname \put(4259,2339){\makebox(0,0)[l]{\strut{}Christoffel sampling}}\csname LTb\endcsname \put(4259,2031){\makebox(0,0)[l]{\strut{}BSS frame subsampling}}\csname LTb\endcsname \put(4259,1724){\makebox(0,0)[l]{\strut{}maximal ETF construction}}\csname LTb\endcsname \put(4259,1417){\makebox(0,0)[l]{\strut{}Singer ETF construction}}\csname LTb\endcsname \put(4259,1110){\makebox(0,0)[l]{\strut{}graph construction}}\csname LTb\endcsname \put(4259,803){\makebox(0,0)[l]{\strut{}dimension argument}}\csname LTb\endcsname \put(463,1655){\rotatebox{-270.00}{\makebox(0,0){\strut{}$n$}}}\csname LTb\endcsname \put(1961,0){\makebox(0,0){\strut{}$\varepsilon$}}}\gplbacktext
    \put(0,0){\includegraphics[width={354.00bp},height={141.00bp}]{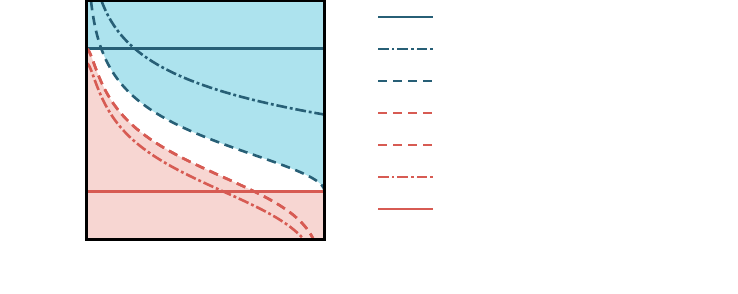}}\gplfronttext
  \end{picture}\endgroup
     \caption{Upper and lower bounds for the number of points in an $L_2$ MZ inequality as a function of the distortion, for $m=998$. The complete-graph, Singer, and maximal ETF lower bounds are compared with the general upper bounds. The Singer and maximal ETF constructions are visually indistinguishable and the $y$-axis is logarithmic.}\label{fig:eps}
\end{figure}

With MZ inequalities being used in function approximation, the results have immediate consequences for least-squares methods.
The number of points needed for the hard-to-discretize spaces is reflected in the condition number of the weighted least-squares matrix.
Let $n_{\rm iter}$ denote the iteration count certified by the standard conjugate-gradient condition-number estimate for a relative target error $0<\delta<1$.
For the hard spaces and $m\le n\le N/2$, every choice of $n$ points and positive weights satisfies
\begin{equation*}
    n_{\rm iter}
    \ge \left\lceil\frac{2\log(\delta/2)}{\log(m/(24n))}\right\rceil,
\end{equation*}
whereas, for every $m$-dimensional space and every integer $n>m$, the constructive subsampling result of \cite{BSU23} gives at most $n$ points for which
\begin{equation*}
    n_{\rm iter}
    \le \left\lceil\frac{2\log(\delta/2)}{\log(m/n)}\right\rceil.
\end{equation*}
These statements concern the standard a priori estimate, see \eqref{eq:iterationbound} later; actual LSQR convergence may be faster because it depends on the full spectrum and on the initial error.
The certified iteration saving is only logarithmic in the oversampling ratio, and therefore large oversampling is not justified solely by this saving when matrix-vector products, function evaluations, or point construction become more expensive with $n$.

This paper is organized as follows.
We establish the connection between $L_2$ MZ inequalities and UNTFs in Section~\ref{sec:untf} and introduce the trace-variance approach.
In Section~\ref{sec:etf}, we investigate the special case of ETFs and prove that they provide the sharpest bounds within the trace-variance approach.
We continue with the complete-graph edge frame in Section~\ref{sec:graph}, which exists for every dimension, and establish the connection between $L_2$ MZ inequalities and spectral graph sparsification.
We conclude with an application in Section~\ref{sec:lsqr}, where we analyze the consequences for the arithmetic cost of the least-squares method for function approximation.

\subsection*{Notation.}
$\preceq$ denotes the Loewner order and
$\|\cdot\|_{\rm HS}$ the Hilbert--Schmidt norm.
 \section{Matrix formulation and the trace-variance obstruction}\label{sec:untf}

\subsection{MZ inequalities in terms of Gram matrices}

Let $\{\eta_1,\dots,\eta_m\}$ be an orthonormal basis of $\mathcal F$.
Using the evaluation vector
\begin{equation*}
    \bm \eta(x) = [\eta_1(x), \dots, \eta_m(x)]^\ast \in\mathds C^m,
\end{equation*}
every $f\in\mathcal F$ can be represented as $f_{\bm c} = \sum_{k=1}^{m} c_k \eta_k$, with $\|f_{\bm c}\|_{L_2} = \|\bm c\|_2$, and $f_{\bm c}(x) = \bm\eta(x)^\ast\bm c$.
For a set of points and weights, the discrete Gram matrix is given by
\begin{equation}\label{eq:M}
    \bm M \coloneqq \sum_{i=1}^{n}w_i\bm\eta(x_i)\bm\eta(x_i)^\ast.
\end{equation}
Then $\sum_{i=1}^{n}w_i|f_{\bm c}(x_i)|^2 = \bm c^\ast\bm M\bm c$, so \eqref{eq:mz} is equivalent to
\begin{equation*}
    (1-\varepsilon)\Id_m
    \preceq \bm M
    \preceq (1+\varepsilon)\Id_m.
\end{equation*}
This elementary reformulation will be used repeatedly.
The rank of $\bm M$ is at most $n$, which gives the universal lower bound $n\ge m$ for $\varepsilon<1$.
To obtain non-trivial lower bounds, we must exploit the geometry of the rank-one matrices $\bm\eta(x)\bm\eta(x)^\ast$.

\subsection{Finite function spaces generated by tight frames}

Let $\mathds F\in\{\mathds R,\mathds C\}$ and $\mathcal H\subseteq\mathds F^M$ be an $m$-dimensional subspace over $\mathds F$.
A family $\Phi = \{\bm\varphi_1,\dots,\bm\varphi_N\}\subset\mathcal H$ is called a \emph{unit-norm tight frame} (UNTF) if $\|\bm\varphi_1\|_2 = \dots = \|\bm\varphi_N\|_2 = 1$ and
\begin{equation}
    \sum_{i=1}^{N} | \langle \bm a,\bm\varphi_i\rangle|^2
    = \frac Nm \| \bm a\|_2^2
    \quad\text{for all } \bm a\in\mathcal H .
\end{equation}
The frame bound is forced by taking the trace of the frame operator.

\begin{proposition}[Tight frames as function spaces]\label{framemz}
    Let $\Phi\subset\mathcal H$ be a UNTF for $\mathcal H\subseteq\mathds F^M$ an $m$-dimensional subspace.
    On the finite set $\Omega=\{1,\dots,N\}$ equipped with the counting measure, we define
    \begin{equation}\label{eq:framespace}
        \mathcal F_\Phi
        :=\Big\{
        f_{\bm a}(i)=\sqrt{\frac mN}\,
        \langle\bm a,\bm\varphi_i\rangle:
        \bm a\in\mathcal H
        \Big\}.
    \end{equation}
    Then $\mathcal F_\Phi$ is an $m$-dimensional subspace of $L_2$ and the map $\bm a\mapsto f_{\bm a}$ is an isometry.
    Moreover, points $J\subseteq\{1,\dots,N\}$ and weights $w_i>0$ for $i\in J$ satisfy \eqref{eq:mz} on $\mathcal F_\Phi$ if and only if
    \begin{equation}\label{eq:framesub}
        \frac Nm(1-\varepsilon)\|\bm a\|_2^2
        \le \sum_{i\in J}w_i|\langle\bm a,\bm\varphi_i\rangle|^2
        \le \frac Nm(1+\varepsilon)\|\bm a\|_2^2
        \quad\text{for all }\bm a\in\mathcal H,
    \end{equation}
    i.e., the family $\{\sqrt{w_i}\bm\varphi_i\}_{i\in J}$ forms a frame.
\end{proposition}

\begin{proof}
    Tightness of the frame gives for all $\bm a\in\mathcal H$
    \begin{equation*}
        \|f_{\bm a}\|_{L_2}^2
        = \sum_{i=1}^N|f_{\bm a}(i)|^2
        =\frac mN\sum_{i=1}^N
        |\langle\bm a,\bm\varphi_i\rangle|^2
        =\|\bm a\|_{2}^2.
    \end{equation*}
    For the weighted sum we obtain
    \begin{equation*}
        \sum_{i\in J}w_i|f(i)|^2
        = \frac mN \sum_{i\in J}w_i|\langle\bm a,\bm\varphi_i\rangle|^2 .
    \end{equation*}
    Since the quantities involved are the same, the assertion follows.
\end{proof}

Thus $L_2$ MZ inequalities for $\mathcal F_\Phi$ correspond to subframes of $\Phi$.
For the function space $\mathcal F_\Phi$, the tightness of the frame corresponds to the orthonormality of a basis.
The unit-norm property translates to the Christoffel function being constant.
Indeed, for $\{\bm e_k\}_{k=1}^{m}$ an orthonormal basis of $\mathcal H$, we have for $i\in\Omega$
\begin{equation*}
    \sum_{k=1}^{m} |f_{\bm e_k}(i)|^2
    = \frac mN \|\bm\varphi_i\|_2^2
    = \frac mN.
\end{equation*}
A constant Christoffel function is not assumed for general function spaces, but it removes variation in the norms of the individual evaluation vectors.

\subsection{Centered projectors and trace variance}

To facilitate computations with UNTFs, we introduce \emph{projectors}
\begin{equation*}
    P_i\colon\mathcal H\to\mathcal H,\quad \bm a\mapsto\langle \bm a,\bm\varphi_i\rangle\bm\varphi_i .
\end{equation*}
Tightness of the initial frame $\Phi$ becomes equivalent to $\sum_{i=1}^{N} P_i = \frac Nm \, \Id_{\mathcal H}$.
Further, the frame operator of the subframe $\{\sqrt{w_i}\bm\varphi_i\}_{i\in J}$ becomes
\begin{equation}\label{eq:subframeop}
    S_{J,w} = \sum_{i\in J}w_iP_i .
\end{equation}
Centering the projectors yields traceless self-adjoint operators on $\mathcal H$
\begin{equation}\label{eq:Q}
    Q_i \coloneqq P_i - \frac 1m {\rm Id}_{\mathcal H},
\end{equation}
which span a real Hilbert space with the Hilbert--Schmidt inner product.
Their Hilbert--Schmidt Gram matrix is
\begin{equation}\label{eq:gram}
    \bm G_\Phi = \Big[\trace(Q_iQ_j)\Big]_{i,j=1}^{N} \in\mathds R^{N\times N} .
\end{equation}
We have $\sum_{i=1}^{N} Q_i = 0$, and hence $\bm G_\Phi\bm 1_N = \bm 0_N$.
A central quantity will be the smallest eigenvalue in the orthogonal complement of the eigenvector $\bm 1_N$, which we denote by
\begin{equation}\label{eq:gamma}
    \gamma_\Phi\coloneqq\lambda_{\min}(\bm G_\Phi |_{\bm 1_N^\perp}) .
\end{equation}
The condition $\gamma_\Phi>0$ says that the only linear dependence among the centered projectors $Q_i$ is their zero sum.

\begin{lemma}\label{tracevariance}
    Let $\Phi\subset\mathcal H\subset\mathds F^M$ be a UNTF with $\gamma_\Phi>0$ from \eqref{eq:gamma}.
    Further, let $J\subseteq\{1,\dots,N\}$ be non-empty and $w_i>0$ for $i\in J$.
    For the subsampled frame operator $S_{J,w}$ from \eqref{eq:subframeop} we have
    \begin{equation*}
        \frac{\|S_{J,w}\|_{\rm HS}^2}{(\trace S_{J,w})^2} 
        \ge \frac 1m + \gamma_\Phi\Big(\frac{1}{|J|}-\frac 1N\Big) .
    \end{equation*}
\end{lemma}

\begin{proof}
    We extend $w$ by zero outside of $J$ and define $\bm w = (w_1,\dots,w_N)^\top$.
    Since each $P_i$ has trace one, we obtain
    \begin{equation*}
        t 
        \coloneqq \trace S_{J,w}
        = \sum_{i=1}^{N} w_i \trace(P_i)
        = \sum_{i=1}^{N} w_i \| \bm\varphi_i \|_2^2
        = \|\bm w\|_1 .
    \end{equation*}
    The identity ${\rm Id}_{\mathcal H}$ is Hilbert--Schmidt orthogonal to the centered projectors $Q_i$, and therefore
    \begin{equation}\label{eq:lkfdsoj}
        \|S_{J,w}\|_{\rm HS}^2
        = \Big\| \frac tm {\rm Id}_{\mathcal H} + \sum_{i=1}^{N} w_i Q_i \Big\|_{\rm HS}^2
        = \frac{t^2}{m} + \bm w^\top\bm G_\Phi\bm w .
    \end{equation}
    Using $\bm G_\Phi\bm 1_N = \bm 0$ and the definition of $\gamma_\Phi$, we obtain
    \begin{align*}
        \bm w^\top\bm G_\Phi\bm w
        &= \Big(\bm w-\frac tN\bm 1_N\Big)^\top \bm G_\Phi \Big(\bm w-\frac tN\bm 1_N\Big) \\
        &\ge \gamma_\Phi \Big(\|\bm w\|_2^2-\frac{t^2}{N}\Big) .
    \end{align*}
    Applying Cauchy--Schwarz on the $|J|$ nonzero entries gives $\|\bm w\|_2^2 \ge t^2/|J|$.
    Substituting this into \eqref{eq:lkfdsoj} and dividing by $t^2$ proves the assertion.
\end{proof}

Note that, among positive semidefinite operators, the quantity $\|\cdot\|_{\rm HS}^2/(\trace(\cdot))^2$ equals $m^{-1}$ exactly for scalar multiples of the identity.
Lemma~\ref{tracevariance} says that a small support $|J|$ forces an increase of this fraction because the centered projectors $Q_i$ cannot cancel sufficiently well.

\begin{lemma}[Kantorovich]\label{kantorovich}
    Let $0\le\varepsilon<1$ and let $T$ be positive definite on an $m$-dimensional Hilbert space with
    \begin{equation*}
        (1-\varepsilon)\Id
        \preceq T
        \preceq (1+\varepsilon)\Id.
    \end{equation*}
    Then
    \begin{equation*}
        \frac{\|T\|_{\rm HS}^2}{(\trace T)^2}
        \le \frac{1}{m(1-\varepsilon^2)} .
    \end{equation*}
\end{lemma}

\begin{proof}
    Let $\lambda_1,\dots,\lambda_m$ be the eigenvalues of $T$.
    The Kantorovich inequality, cf.~\cite[Lemma on page~142]{Kan48}, states
    \begin{equation*}
        \sum_{k=1}^{m} \alpha_k \lambda_k
        \sum_{k=1}^{m} \alpha_k \lambda_k^{-1}
        \le \frac{(A+B)^2}{4AB} \Big(\sum_{k=1}^m \alpha_k\Big)^2
    \end{equation*}
    for nonnegative $\alpha_k$ and $0<A\le \lambda_k \le B$, $k=1,\dots,m$.
    We set $A = 1-\varepsilon$, $B = 1+\varepsilon$, and $\alpha_k = \lambda_k$.
    The three sums become
    \begin{equation*}
        \sum_{k=1}^m \alpha_k\lambda_k=\|T\|_{\rm HS}^2,
        \quad
        \sum_{k=1}^m \alpha_k\lambda_k^{-1}=m,
        \quad\text{and}\quad
        \sum_{k=1}^m \alpha_k=\trace T,
    \end{equation*}
    and the assertion follows.
\end{proof}

\begin{theorem}[General UNTF lower bound]\label{untfbound}
    Let $\Phi = \{\bm\varphi_1,\dots,\bm\varphi_N\}\subset\mathcal H\subseteq\mathds F^M$ be a UNTF in an $m$-dimensional subspace $\mathcal H$ with $\gamma_\Phi>0$ from \eqref{eq:gamma}.
    Suppose $J\subseteq\{1,\dots,N\}$ and $w_i>0$ for $i\in J$ satisfy
    \begin{equation*}
        (1-\varepsilon)\|\bm a\|_2^2
        \le \sum_{i\in J}w_i|\langle\bm a,\bm\varphi_i\rangle|^2
        \le (1+\varepsilon)\|\bm a\|_2^2
        \quad\text{for all }\bm a\in\mathcal H
    \end{equation*}
    for some $0\le\varepsilon<1$.
    Then
    \begin{equation*}
        |J| \ge \max\Big\{m, \frac{m\gamma_\Phi N(1-\varepsilon^2)}{m\gamma_\Phi(1-\varepsilon^2)+N\varepsilon^2} \Big\} .
    \end{equation*}
\end{theorem}

\begin{proof}
    The frame condition can equivalently be expressed in terms of the subsampled frame operator $S_{J,w}$ from \eqref{eq:subframeop}
    \begin{equation*}
        (1-\varepsilon){\rm Id}_{\mathcal H}
        \preceq S_{J,w}
        \preceq (1+\varepsilon){\rm Id}_{\mathcal H} .
    \end{equation*}
    The rank of $|J|$ rank-one operators is at most $|J|$, whereas $S_{J,w}$ has rank $m$.
    Hence $|J|\ge m$.
    Combining Lemmas~\ref{tracevariance} and \ref{kantorovich} gives
    \begin{equation*}
        \frac 1m + \gamma_\Phi\Big(\frac{1}{|J|} - \frac 1N\Big)
        \le \frac{\|S_{J,w}\|_{\rm HS}^2}{(\trace S_{J,w})^2}
        \le \frac{1}{m(1-\varepsilon^2)} .
    \end{equation*}
    Solving for $|J|$ yields the second term in the assertion.
\end{proof}

For an $L_2$ MZ inequality on $\mathcal F_\Phi$, the frame inequality \eqref{eq:framesub} is obtained after replacing the MZ weights $w_i$ by $(m/N)w_i$.
This global rescaling does not change the support $|J|$.

\subsection{The exact complex endpoint}

The trace-variance method gives sharp order but does not by itself produce the exact value $m^2$ in every dimension for complex function spaces.
The case $\varepsilon=0$ follows from a different construction based on Hermitian matrices.

\begin{lemma}\label{lem:hermitian-projector-basis}
    For every $m\geq1$, there are $m^2$ vectors in $\mathds C^m$ whose rank-one projectors form a real basis of the Hermitian $m\times m$ matrices.
\end{lemma}

\begin{proof}
    Let $\{\bm e_i\}_{i=1}^m$ be the standard basis and take
    \begin{equation}\label{eq:exact-vectors}
        \{\bm e_i:1\leq i\leq m\}
        \cup
        \{\bm e_i+\bm e_j:1\leq i<j\leq m\}
        \cup
        \{\bm e_i+\mathrm i\bm e_j:1\leq i<j\leq m\}.
    \end{equation}
    There are $m+2\binom m2=m^2$ vectors.
    In a real linear combination of their projectors, the $(i,j)$ entry with $i<j$ is of the form $\beta_{ij}-\mathrm i\gamma_{ij}$, where $\beta_{ij}$ and $\gamma_{ij}$ are the coefficients of the second and third families.
    If the combination vanishes, all $\beta_{ij}$ and $\gamma_{ij}$ vanish.
    The diagonal entries then force the coefficients of the projectors $\bm e_i\bm e_i^*$ to vanish.
    Hence the projectors are linearly independent, and their number equals the real dimension $m^2$ of the Hermitian matrices.
\end{proof}

\begin{theorem}[Complex endpoint]\label{thm:exact-endpoint}
    For every $m\geq1$ it holds $\mathcal N(m,0)=m^2$.
\end{theorem}

\begin{proof}
    Let $\Omega$ be the set of vectors in \eqref{eq:exact-vectors}, equipped with counting measure, and define
    \begin{equation*}
        \mathcal F\coloneqq\{f_{\bm a}:\Omega\to\mathds C:\ f_{\bm a}(\bm v)=\langle\bm a,\bm v\rangle,\ \bm a\in\mathds C^m\}.
    \end{equation*}
    This space has dimension $m$ because $\Omega$ contains the standard basis.
    Put $P_{\bm v}=\bm v\bm v^*$ and $T=\sum_{\bm v\in\Omega}P_{\bm v}$.
    An exact weighted formula supported on $J\subseteq\Omega$ would imply
    \begin{equation*}
        \sum_{\bm v\in J}w_{\bm v}P_{\bm v}=T
        =\sum_{\bm v\in\Omega}P_{\bm v}.
    \end{equation*}
    By Lemma~\ref{lem:hermitian-projector-basis}, the expansion in these projectors is unique.
    Therefore $J=\Omega$ and every weight equals one.
    Thus this space requires $m^2$ points.
    The matching upper bound is due to \cite{BKPSU25, ST25}.
\end{proof}

\begin{remark}[Real endpoint]\label{rem:real-endpoint}
    If one defines the analogous worst-case quantity for real function spaces, the exact value is $m(m+1)/2$.
    The upper bound follows from the dimension of the real quadratic product space, i.e., the dimension of real symmetric matrices.
    For the lower bound we may use the vectors $\bm e_i$ and $\bm e_i+\bm e_j$, whose projectors form a basis of the real symmetric matrices.
    The lower bound also follows from the complete graph construction in Theorem~\ref{spectralgraphbound}.
\end{remark}
 \section{ETF-based construction}\label{sec:etf}

A family of unit vectors $\Phi=\{\bm\varphi_1,\dots,\bm\varphi_N\}\subset\mathds C^m$ is \emph{equiangular} if $|\langle\bm\varphi_i,\bm\varphi_j\rangle|$ is constant for $i\ne j$.
A unit-norm equiangular tight frame is abbreviated \emph{ETF}.
The symmetry of their centered projectors makes ETFs natural candidates for Theorem~\ref{untfbound}.
Written in terms of their centered projectors $Q_i$ defined in \eqref{eq:Q}, we have
\begin{align}
    \langle Q_i,Q_j\rangle_{\rm HS}
    &= \trace(Q_i Q_j) \nonumber\\
    &= \trace(P_iP_j) - \frac 1m \trace P_i - \frac 1m \trace P_j + \frac{1}{m^2} \trace\Id_{\mathcal H} \nonumber\\
    &= |\langle\bm\varphi_i,\bm\varphi_j\rangle|^2 - \frac 1m \,, \label{eq:simplex}
\end{align}
which is constant for $i\ne j$ by the ETF property.
Together with $\sum_{i=1}^{N}Q_i = 0$, this implies that the centered projectors of an ETF form a regular simplex in the space of Hermitian matrices.

\subsection{General ETF lower bound}

For $N$ unit vectors spanning an $m$-dimensional space, the Welch bound bounds their \emph{coherence}
\begin{equation}\label{eq:welch}
    \mu^2
    :=\max_{i\ne j}
      |\langle\bm\varphi_i,\bm\varphi_j\rangle|^2
    \ge \frac{N-m}{m(N-1)}.
\end{equation}
Equality holds exactly for ETFs, see e.g., \cite[Theorem~1.10.1]{Christensen16}.

\begin{lemma}\label{etfcenteredgram}
    Let $\mathds F\in\{\mathds R,\mathds C\}$ and let $\Phi = \{\bm\varphi_1,\dots,\bm\varphi_N\}\subset\mathds F^m$ be a UNTF.
    Let $\bm G_\Phi$ be the Gram matrix of the centered projectors $Q_i$ from \eqref{eq:gram}.
    Then $\Phi$ is an ETF with coherence $\alpha$ if and only if
    \begin{equation*}
        \bm G_\Phi
        = (1-\alpha^2)
        \Big(\bm I_N-\frac1N\bm 1_N\bm 1_N^\top\Big).
    \end{equation*}
    In that case, $\gamma_\Phi$ from \eqref{eq:gamma} equals
    \begin{equation*}
        \gamma_\Phi = 1-\alpha^2 = \frac{N(m-1)}{m(N-1)}.
    \end{equation*}
\end{lemma}

\begin{proof}
    Suppose $\{\bm\varphi_1,\dots,\bm\varphi_N\}$ is an ETF.
    Because of \eqref{eq:simplex} and the equality in the Welch bound \eqref{eq:welch} we have
    \begin{equation*}
        [\bm G_\Phi]_{i,j}
        = \trace(Q_iQ_j)
        = \begin{cases}
            1-\frac 1m = (1-\alpha^2)\Big(1-\frac 1N\Big) & \text{for }i=j,\\
            \alpha^2-\frac 1m = -\frac{1-\alpha^2}{N} & \text{for }i\ne j.
        \end{cases}
    \end{equation*}
    This proves that $\bm G_\Phi$ has the stated form.
    The value for $\gamma_\Phi$ follows from the equality case of the Welch bound \eqref{eq:welch} for ETFs.

    The equiangularity follows from the off-diagonal entries of $\bm G_\Phi$ being constant and $\trace(Q_iQ_j) = |\langle\bm\varphi_i,\bm\varphi_j\rangle|^2-\frac1m$.
\end{proof}

\begin{theorem}\label{etfbound}
    Let $\Phi=\{\bm\varphi_1,\dots,\bm\varphi_N\}\subset\mathds C^m$ be an ETF.
    Any set of $n$ points and positive weights satisfying \eqref{eq:mz} on the $m$-dimensional space $\mathcal F_{\Phi}$ from \eqref{eq:framespace} with distortion $0\le \varepsilon<1$ obeys
    \begin{equation*}
        n\ge
        \max\Big\{
        m,
        \frac{N(1-\varepsilon^2)}
             {1+\frac{N-m}{m-1}\varepsilon^2}
        \Big\}.
    \end{equation*}
\end{theorem}

\begin{proof}
    By Lemma~\ref{etfcenteredgram} we have $\gamma_\Phi = \frac{N(m-1)}{m(N-1)}$.
    Plugging this in Theorem~\ref{untfbound} and simplifying yields the assertion.
\end{proof}

The next result states the extremal role of ETFs within the trace-variance approach.

\begin{theorem}\label{simplexextremality}
    Let $\Phi = \{\bm\varphi_1,\dots,\bm\varphi_N\}\subset\mathcal H\subseteq\mathds F^M$ be a UNTF spanning an $m$-dimensional subspace $\mathcal H$, with $\gamma_\Phi>0$ from \eqref{eq:gamma}.
    Then
    \begin{equation*}
        N\le \begin{cases}m(m+1)/2 & \text{for }\mathds F=\mathds R\,,\\ m^2&\text{for }\mathds F=\mathds C\end{cases}
    \end{equation*}
    and
    \begin{equation*}
        \gamma_\Phi
        \le \frac{N(m-1)}{m(N-1)}.
    \end{equation*}
    Equality holds if and only if $\Phi$ is an ETF.
\end{theorem}

\begin{proof}
    Since the kernel of the Gram matrix $\bm G_\Phi$ is $\Span\{\bm 1_N\}$ and it is positive definite on $\bm 1_N^\perp$, its rank equals $N-1$.
    At the same time the rank of $\bm G_\Phi$ is at most the real dimension of the traceless self-adjoint operators on $\mathds F^m$, which is $m(m+1)/2-1$ in the real case and $m^2-1$ in the complex case.
    This proves the first assertion.

    The Gram matrix has $N-1$ eigenvalues on $\bm 1_N^\perp$, all at least $\gamma_\Phi$.
    Thus,
    \begin{equation*}
        (N-1)\gamma_\Phi \le \trace\bm G_\Phi
        = \sum_{i=1}^N\|Q_i\|_{\rm HS}^2
        = N\Big(1-\frac1m\Big) ,
    \end{equation*}
    which proves the second assertion.
    Equality holds precisely when all nonzero eigenvalues of $\bm G_\Phi$ agree, that is,
    \begin{equation*}
        \bm G_\Phi
        =\gamma_\Phi
        \Big(\bm I_N-\frac1N\bm 1_N\bm 1_N^\top\Big).
    \end{equation*}
    Due to Lemma~\ref{etfcenteredgram} this characterizes an ETF.
\end{proof}

For fixed $m$ and $N$, Theorem~\ref{simplexextremality} shows that ETFs maximize the spectral gap $\gamma_\Phi$ and hence maximize the lower bound furnished by Theorem~\ref{untfbound}.
Moreover, the resulting ETF lower bound is increasing in $N$, so among the ETFs available in a given dimension the largest frames give the strongest obstruction.

\subsection{Maximal complex ETFs}

The restriction $N\le m^2$ in Theorem~\ref{simplexextremality} is known as the complex Gerzon bound, i.e., there are at most $m^2$ equiangular lines in $\mathds C^m$.
A maximal complex ETF has exactly $N=m^2$ vectors.
After scaling its rank-one projectors by $1/m$, it is a symmetric informationally complete positive operator-valued measure (SIC-POVM) in quantum information.
Zauner conjectured that such a system exists in every dimension, cf.~\cite{Zauner99}.
So far, only a finite number of maximal ETFs are known, and the conjecture remains open.
The history and known constructions of such frames are surveyed in \cite{FM15, FHS17}.

Plugging in maximal ETFs in Theorem~\ref{etfbound} yields the following statement.

\begin{corollary}[Maximal ETF bound]\label{zaunerbound}
    Suppose that a maximal ETF exists in $\mathds C^m$.
    Any set of $n$ points and positive weights satisfying \eqref{eq:mz} on the $m$-dimensional space $\mathcal F_{\Phi}$ from \eqref{eq:framespace} with distortion $0\le \varepsilon<1$ obeys
    \begin{equation}\label{eq:zaunerbound}
        n\ge
        \max\Big\{
        m,
        \frac{m^2(1-\varepsilon^2)}
             {1+m\varepsilon^2}
        \Big\}.
    \end{equation}
    Conditional on Zauner's conjecture, this holds in every dimension.
\end{corollary}

\begin{proof}
    Set $N=m^2$ in Theorem~\ref{etfbound}.
\end{proof}

\subsection{Singer ETFs}

A deterministic family close to the complex Gerzon bound is available in infinitely many dimensions. A $(v,k,\lambda)$ difference set in a finite group $\Gamma$ is a subset $D\subset\Gamma$ of cardinality $k$ such that every nonidentity element has exactly $\lambda$ representations as $d_1d_2^{-1}$ with $d_1,d_2\in D$. Singer's construction gives a cyclic
\begin{equation*}
    (q^2+q+1,q+1,1)
\end{equation*}
difference set whenever $q$ is a prime power~\cite{Singer38}.

Let
\begin{equation*}
    M=q^2+q+1,
    \quad
    \omega=\exp(2\pi\mathrm i/M),
\end{equation*}
identify the cyclic group with $\mathds Z_M$, and let $D\subset\mathds Z_M$ be a Singer difference set. Define
\begin{equation}\label{eq:singer-frame}
    \bm\varphi_i
    =\frac1{\sqrt{q+1}}\big(\omega^{ik}\big)_{k\in D},
    \quad i\in\mathds Z_M.
\end{equation}
These vectors are the normalized columns of the Fourier submatrix obtained by restricting the rows indices to $D$; orthogonality of those rows shows tightness.
To show equiangularity, let $i\neq j$.
The difference-set property yields
\begin{align*}
    |\langle\bm\varphi_i,\bm\varphi_j\rangle|^2
    &=\frac1{(q+1)^2}
      \sum_{k,\ell\in D}\omega^{(i-j)(k-\ell)}\\
    &=\frac1{(q+1)^2}
      \Big((q+1)+\sum_{r=1}^{M-1}\omega^{(i-j)r}\Big)\\
    &=\frac{q}{(q+1)^2}.
\end{align*}
Thus \eqref{eq:singer-frame} is an ETF with
\begin{equation}\label{eq:singer-parameters}
    m=q+1,
    \quad
    N=q^2+q+1=m^2-m+1.
\end{equation}
This harmonic-frame construction is also discussed in \cite[Section~2.1.2]{SH03}.

\begin{corollary}[Singer-ETF bound]\label{cor:singer-etf}
    Let $m-1$ be a prime power and $\Phi$ be the Singer ETF.
    Any set of $n$ points and positive weights satisfying \eqref{eq:mz} on the $m$-dimensional space $\mathcal F_{\Phi}$ from \eqref{eq:framespace} with distortion $0\le \varepsilon<1$ obeys
    \begin{equation*}
        n
        \ge \max\Big\{
            m,
            \frac{(m^2-m+1)(1-\varepsilon^2)}{1+(m-1)\varepsilon^2}
        \Big\}.
    \end{equation*}
\end{corollary}

\begin{proof}
    Use \eqref{eq:singer-parameters} in Theorem~\ref{etfbound}.
\end{proof}

The recent Singer--Zauner gap theorem~\cite{FJM26} shows that no complex ETF can have a cardinality strictly between $m^2-m+1$ and $m^2$.
Hence the Singer and maximal cases cover the two largest possible ETF cardinalities.
This fact is not needed for the lower bounds, but it clarifies why the two constructions are the natural near-extremal instances of Theorem~\ref{etfbound}.
 \section{Graph-based construction}\label{sec:graph}

Since ETFs with quadratically many elements are not known in every dimension, the edge frame of a complete graph supplies a uniform substitute.
It is not equiangular, but its centered-projector gap $\gamma_\Phi$ is bounded below by an absolute constant.
While the presented construction is done for $\mathds F=\mathds R$, all results hold with $\mathds F=\mathds C$ for the underlying field as well.

\subsection{The complete-graph edge frame}

Let $\mathcal K=(V,\binom V2)$ be the complete graph on $|V| = m+1 \ge3$ vertices, and let
\begin{equation*}
    \mathcal H_{\mathcal K}=\bm 1_{|V|}^\perp\subset\mathds R^{|V|}.
\end{equation*}
In this section, the subspace $\mathcal H$ from Section~\ref{sec:untf} is $\mathcal H_{\mathcal K}$ with dimension $|V|-1 = m$.
The graph Laplacian of the complete graph is
\begin{equation*}
    \bm L_{\mathcal K}
    = \bm D - \bm A
    = |V|\bm I - \bm 1_{|V|}\bm 1_{|V|}^\top ,
\end{equation*}
where $\bm D$ is the degree matrix and $\bm A$ is the adjacency matrix.
We have $|V|\Id_{\mathcal H_{\mathcal K}} = \bm L_{\mathcal K}|_{\mathcal H_{\mathcal K}}$, where the restriction is regarded as an endomorphism of $\mathcal H_{\mathcal K}$.
For the edges $e=\{u,v\}\in\binom{V}{2}$ of the complete graph, we define
\begin{equation}\label{eq:edgeframe}
    \bm\varphi_e
    =\frac{\bm e_u-\bm e_v}{\sqrt2}\in \mathcal H_{\mathcal K}.
\end{equation}
where $\bm e_u$ denotes the unit vector in $\mathds R^{|V|}$.
Note that the vectors $\bm\varphi_e$ do not act on the full space $\mathds R^{|V|}$.
There are $N=\binom {|V|}2=\frac{m(m+1)}2$ such vectors.

\begin{lemma}[Spectrum of the centered edge projectors]\label{edgeframespectrum}
    The vectors $\{\bm\varphi_e : e\in\binom{V}{2}\}$ from \eqref{eq:edgeframe} form a UNTF for $\mathcal H_{\mathcal K}$.
    The centered-projector Gram matrix $\bm G_{\mathcal K}$ from \eqref{eq:gram} has eigenvalue zero on $\bm 1_N$, and, on $\bm 1_N^\perp$, the eigenvalues
    \begin{equation*}
        \frac {|V|}4 \text{ with multiplicity }|V|-1
        \quad\text{and}\quad
        \frac12 \text{ with multiplicity }\frac{|V|(|V|-3)}2.
    \end{equation*}
    Consequently,
    \begin{equation*}
        \gamma_{\mathcal K}
        =\begin{cases}
            3/4,&|V|=3,\\
            1/2,&|V|\ge4.
        \end{cases}
    \end{equation*}
\end{lemma}

\begin{proof}
    For the projector $P_e\colon\mathcal H_{\mathcal K}\to\mathcal H_{\mathcal K}, \bm a\mapsto \langle\bm a,\bm\varphi_e\rangle\bm\varphi_e$, we obtain
    \begin{equation*}
        \sum_{e\in\binom V2}P_e
        =\frac12\bm L_{\mathcal K}\big|_{\mathcal H_{\mathcal K}}
        =\frac {|V|}2\Id_{\mathcal H_{\mathcal K}}
        =\frac Nm\Id_{\mathcal H_{\mathcal K}}.
    \end{equation*}
    Since the edge vectors $\bm\varphi_e$ are unit-norm by definition, they form a UNTF.

    Let $\bm A_{\mathrm L} \in\mathds R^{N\times N}$ be the adjacency matrix of the line graph of $\mathcal K$.
    Let
    \begin{equation*}
        \bm B\in\mathds R^{|V|\times N}
        \quad\text{with entries}\quad
        b_{v,e} = \begin{cases}1 & \text{for } v\in e\,,\\0& \text{otherwise.}\end{cases}
    \end{equation*}
    be the unsigned vertex-edge incidence matrix of $\mathcal K$.
    Because every edge contains two vertices and two distinct edges contribute one precisely when they meet, we have the relation
    \begin{equation*}
        \bm B^\top\bm B
        = \Big[ \sum_{v\in V} \delta_{v\in e}\delta_{v\in f} \Big]_{e,f\in \binom{V}{2}}
        = 2\bm I_N+\bm A_{\mathrm L} .
    \end{equation*}
    On the other hand, we have
    \begin{equation*}
        \bm B\bm B^\top
        = \Big[ \sum_{e\in\binom{V}{2}} \delta_{v\in e}\delta_{u\in e} \Big]_{u,v\in V}
        = (|V|-2)\bm I_{|V|}+\bm 1_{|V|}\bm 1_{|V|}^\top .
    \end{equation*}
    The latter matrix has eigenvalues $2(|V|-1)$ with multiplicity one and $|V|-2$ with multiplicity $|V|-1$.
    Since $\bm B^\top\bm B$ has the same nonzero eigenvalues, the spectrum of the line-graph matrix $\bm A_{\mathrm L}$ is
    \begin{center}
    \begin{tabular}{c|ccc}
        eigenvalue & $2(|V|-2)$ & $|V|-4$ & $-2$ \\
        \hline
        multiplicity & $1$ & $|V|-1$ & $N-|V|$
    \end{tabular}
    \end{center}

    Now we relate the spectrum of $\bm A_{\mathrm L}$ to the centered-projector Gram matrix $\bm G_{\mathcal K}$.
    For the Gram matrix of the projectors we have the entries
    \begin{equation*}
        \trace(P_eP_f)
        = |\langle\bm\varphi_e,\bm\varphi_f\rangle|^2
        = \begin{cases} 1 & \text{for }e=f \,, \\ 1/4 & \text{if two edges } e\neq f \text{ meet in one vertex,}\\ 0 & \text{if the edges } e\text{ and } f \text{ are disjoint.}\end{cases}
    \end{equation*}
    Therefore $[\trace(P_eP_f)]_{e,f\in\binom{V}{2}} = \bm I_{N} + \frac 14 \bm A_{\mathrm L}$ has the eigenvalues
    \begin{center}
    \begin{tabular}{c|ccc}
        eigenvalue & $|V|/2$ & $|V|/4$ & $1/2$ \\
        \hline
        multiplicity & $1$ & $|V|-1$ & $N-|V|$
    \end{tabular}
    \end{center}
    Because every edge meets $2(|V|-2)$ other edges, we obtain
    \begin{equation*}
        \bm A_{\mathrm L}\bm 1_{N} = 2(|V|-2)\bm 1_{N}
        \quad\Leftrightarrow\quad
        [\trace(P_eP_f)]_{e,f\in\binom{V}{2}}\bm 1_{N}
        = \frac{|V|}{2}\bm 1_{N} ,
    \end{equation*}
    i.e., this determines the one-dimensional eigenspace for the eigenvalue $|V|/2$.
    To obtain the Gram matrix for the centered projectors, we have to subtract $m^{-1}\bm 1_N\bm 1_N^\top$
    \begin{equation*}
        \bm G_{\mathcal K}
        = \bm I_{N} + \frac 14 \bm A_{\mathrm L} - \frac 1m\bm 1_{N}\bm 1_{N}^\top .
    \end{equation*}
    This cancels the eigenvalue ${|V|}/2$ in the direction $\bm 1_N$ and leaves the other eigenvalues unchanged, proving the assertion.
\end{proof}

\begin{remark}\label{twoperpspaces}
    Two different orthogonal complements appear.
    The vertex space $\mathcal H_{\mathcal K}=\bm 1_{|V|}^\perp$ removes the common kernel of graph Laplacians and is the Hilbert space on which the frame acts.
    The coefficient space $\bm 1_N^\perp$ removes uniform reweighting of the frame elements and is where the centered-projector gap is measured.
\end{remark}

\subsection{Support lower bound for spectral sparsifiers}

Let $\mathcal G=(V,E,w)$ be a weighted graph with $E\subseteq\binom{V}{2}$.
The weight function $w:E\to(0,\infty)$ is defined on the edges, and we set $w_e=0$ for $e\notin E$.
Its weighted Laplacian is
\begin{equation}\label{eq:laplacian}
    \bm L_{\mathcal G}
    =\bm D_w-\bm W
    =\sum_{\{u,v\}\in E}w_{\{u,v\}} (\bm e_u-\bm e_v)(\bm e_u-\bm e_v)^\top
\end{equation}
with the weight matrix $\bm W = [w_{\{u,v\}}]_{u,v\in V}$ and the weighted degree matrix $\bm D_w = \diag(\deg_w(u))_{u\in V}$, where $\deg_w(u) = \sum_{v\in V} w_{\{u,v\}}$.

The graph $\mathcal G$ is a \emph{spectral sparsifier} of $\mathcal K$ with distortion $\varepsilon$ if
\begin{equation}\label{eq:spectralsparsifier}
    (1-\varepsilon)\bm L_{\mathcal K}
    \preceq \bm L_{\mathcal G}
    \preceq (1+\varepsilon)\bm L_{\mathcal K} .
\end{equation}

\begin{theorem}\label{spectralgraphbound}
    Every spectral sparsifier of the complete graph $\mathcal K$ with $0\le\varepsilon<1$ satisfies
    \begin{equation*}
        |E|\ge
        \max\left\{
        {|V|}-1,
        \frac{{|V|}({|V|}-1)(1-\varepsilon^2)}
             {2\big(1+({|V|}-1)\varepsilon^2\big)}
        \right\}.
    \end{equation*}
\end{theorem}

\begin{proof}
    On $\mathcal H_{\mathcal K}$, the complete-graph Laplacian equals ${|V|}\Id_{\mathcal H_{\mathcal K}}$.
    Hence \eqref{eq:spectralsparsifier} is equivalent to
    \begin{equation*}
        (1-\varepsilon)\Id_{\mathcal H_{\mathcal K}}
        \preceq \frac1{|V|}\bm L_{\mathcal G}\big|_{\mathcal H_{\mathcal K}}
        =\sum_{e\in E}\frac{2w_e}{{|V|}}P_e
        \preceq (1+\varepsilon)\Id_{\mathcal H_{\mathcal K}},
    \end{equation*}
    which is a normalized subframe inequality.
    It remains to apply Theorem~\ref{untfbound} with $m={|V|}-1$, $N={|V|}({|V|}-1)/2$, and $\gamma_{\mathcal K}\ge 1/2$ from Lemma~\ref{edgeframespectrum}.
\end{proof}

The upper bound is supplied by Batson, Spielman, and Srivastava \cite{BSS14}: every weighted graph admits a spectral sparsifier with $\mathcal O(|V|/\varepsilon^2)$ edges.
Thus Theorem~\ref{spectralgraphbound} is a lower-bound counterpart for the complete graph.

A related stronger estimate on the minimal degree appears in \cite[Proposition~4.2]{BSS14} using more graph-specific tools.
The argument presented here is weaker than the graph-specific minimal-degree obstruction, but it has the advantage of arising directly from the same frame mechanism that applies to arbitrary function spaces.

\subsection{Consequences for $L_2$ MZ inequalities}

We define a function space on the edges of $\mathcal K$ by
\begin{equation}\label{eq:graphspace}
    \mathcal F_{\mathcal K} \coloneqq \Big\{
        f_{\bm a} \colon {\textstyle\binom{V}{2}} \to\mathds R : f_{\bm a}(\{u,v\}) = \frac{a_u-a_v}{\sqrt{|V|}},\quad \bm a\in\mathds R^{|V|} \text{ with } \sum_{u\in V}a_u = 0
    \Big\} ,
\end{equation}
where the edge set carries a counting measure.
The space has dimension $|V|-1=m$ and is precisely the frame space \eqref{eq:framespace} generated by the edge vectors \eqref{eq:edgeframe}.

In the following proposition we show that controlling the distortion in $L_2$ MZ inequalities for $\mathcal F_{\mathcal K}$ corresponds to controlling the spectrum of the weighted Laplacian $\bm L_{\mathcal G}$.

\begin{proposition}\label{graphmz}
    Let $\mathcal G = (V,E,w)$ be a weighted graph with $\bm L_{\mathcal G}$ its weighted graph Laplacian \eqref{eq:laplacian}.
    Further, let $\mathcal F_{\mathcal K}$ be the function space \eqref{eq:graphspace} defined on the edges of the complete graph.
    Then, for $0\le \varepsilon <1$, the weighted graph $\mathcal G$ is a spectral sparsifier of the complete graph $\mathcal K$
    \begin{equation*}
        (1-\varepsilon)\bm L_{\mathcal K} \preceq \bm L_{\mathcal G} \preceq (1+\varepsilon)\bm L_{\mathcal K}
    \end{equation*}
    if and only if the edges $E$ and weights $w$ satisfy the $L_2$ MZ inequality
    \begin{equation*}
        (1-\varepsilon)\|f\|_{L_2}^{2}
        \le \sum_{e\in E} w_e |f(e)|^2
        \le (1+\varepsilon)\|f\|_{L_2}^{2}
        \quad\text{for all }f\in \mathcal F_{\mathcal K}.
    \end{equation*}
\end{proposition}

\begin{proof}
    Let $P_e\colon\mathcal H_{\mathcal K}\to\mathcal H_{\mathcal K}, \bm a\mapsto \langle\bm a,\bm\varphi_e\rangle\bm\varphi_e$ be the projector with $\bm\varphi_e$ defined in \eqref{eq:edgeframe}.
    By \eqref{eq:laplacian} we have
    \begin{equation*}
        \bm L_{\mathcal G} |_{\mathcal H_{\mathcal K}}
        = 2\sum_{e\in E} w_e P_e .
    \end{equation*}
    In particular, $|V|\Id_{\mathcal H_{\mathcal K}} = \bm L_{\mathcal K}|_{\mathcal H_{\mathcal K}} = 2\sum_{e\in\binom{V}{2}}P_e$.
    Thus, spectral graph sparsification is equivalent to
    \begin{equation*}
        |V|(1-\varepsilon)\Id_{\mathcal H_{\mathcal K}}
        \preceq 2\sum_{e\in E} w_e P_e
        \preceq |V|(1+\varepsilon)\Id_{\mathcal H_{\mathcal K}},
    \end{equation*}
    which is the frame condition for the weighted subframe $\{\sqrt{w_e}\bm\varphi_e\}_{e\in E}$ with frame bounds $\frac{|V|}{2}(1-\varepsilon)$ and $\frac{|V|}{2}(1+\varepsilon)$.
    By Proposition~\ref{framemz} this is equivalent to the $L_2$ MZ inequalities using $\frac mN = \frac{2}{m+1} = \frac{2}{|V|}$.
\end{proof}

\begin{corollary}\label{graphmzbound}
    Let $m={|V|}-1$.
    Any set of $n$ points and positive weights satisfying \eqref{eq:mz} on the $m$-dimensional space $\mathcal F_{\mathcal K}$ defined in \eqref{eq:graphspace} with distortion $0\le \varepsilon < 1$ obeys
    \begin{equation}\label{eq:unconditional-mz-lower}
        n\ge
        \max\left\{
        m,
        \frac{m(m+1)(1-\varepsilon^2)}
             {2(1+m\varepsilon^2)}
        \right\}.
    \end{equation}
    At $\varepsilon=0$, all $m(m+1)/2$ edges are necessary, and the bound is sharp.
\end{corollary}

\begin{proof}
    The statement follows directly from Theorem~\ref{spectralgraphbound} and Proposition~\ref{graphmz} using $m = |V|-1$.
\end{proof}

We now complete the proof announced in the introduction.

\begin{proof}[Proof of Theorem~\ref{maintheorem}]
    The lower bound is Corollary~\ref{graphmzbound}.
    The upper bound is the minimum of the bounds given by the constructive estimate \eqref{eq:constructive-upper} and the exact-discretization estimate \eqref{eq:exact-upper}.
    The exact endpoint is Theorem~\ref{thm:exact-endpoint}.

    It remains to verify the explicit constants in the second chain of estimates.
    Let
    \begin{equation*}
        T=\min\left\{m^2,\frac{m}{\varepsilon^2}\right\}.
    \end{equation*}
    If $\varepsilon^2\geq1/2$, then $T\leq2m$, so the lower bound gives $\mathcal N(m,\varepsilon)\geq T/2$.
    If $\varepsilon^2<1/2$, the second term in \eqref{eq:unconditional-mz-lower} satisfies
    \begin{equation*}
        \frac{m(m+1)(1-\varepsilon^2)}{2(1+m\varepsilon^2)}
        \geq\frac{m^2}{4(1+m\varepsilon^2)}
        \geq\frac T8.
    \end{equation*}
    For the upper bound, \eqref{eq:constructive-upper} gives at most $5m/\varepsilon^2$ points.
    If $T=m^2$, use the exact upper bound \eqref{eq:exact-upper}.
\end{proof}
 \section{Stability and arithmetic cost of least-squares approximation}\label{sec:lsqr}

The \emph{least-squares method} is used to approximate functions from data $y_1,\dots,y_n\in\mathds C$ given at sample points $x_1,\dots,x_n\in\Omega$; see \cite{Barteldiss} and the references therein for a survey.
The target function usually belongs to an infinite-dimensional function space that models smoothness properties.
As a first step, one chooses a suitable finite-dimensional function space $\mathcal F$ that approximates the infinite-dimensional space well.
Having chosen such a space, the weighted least-squares approximation is defined by
\begin{equation}\label{eq:lsqr}
    S_{\mathcal F}^{\bm X} \bm y
    = \argmin_{g\in\mathcal F} \sum_{i=1}^{n}w_i|g(x_i)-y_i|^2 .
\end{equation}
Given an orthonormal basis $\{\eta_1,\dots,\eta_m\}$ of $\mathcal F$, define
\begin{equation*}
    \bm L = [\eta_k(x_i)]_{i=1,\dots,n;\,k=1,\dots,m}\in\mathds C^{n\times m}
    \quad\text{and}\quad
    \bm W=\diag(w_1,\dots,w_n).
\end{equation*}
By \eqref{eq:M} we have $\bm L^\ast\bm W\bm L = \sum_{i=1}^{n}w_i\bm\eta(x_i)\bm\eta(x_i)^\ast = \bm M$.
Then
\begin{equation*}
    S_{\mathcal F}^{\bm X} \bm y = \sum_{k=1}^{m}\hat g_k\eta_k
    \quad\text{with}\quad
    \bm M \bm{\hat g} = \bm L^\ast\bm W^{1/2}\bm y,
\end{equation*}
where $\bm{\hat g} = [\hat g_1,\dots,\hat g_m]^\top$ and $\bm y = [y_1,\dots,y_n]^\top$.
In practice, LSQR is applied directly to $\bm W^{1/2}\bm L\widehat{\bm g}\approx \bm W^{1/2}\bm y$, thereby avoiding the explicit formation of the normal matrix $\bm M$.
In exact arithmetic, its coefficient iterates agree with those of conjugate gradients applied to the normal equations while being more stable, cf.~\cite{PS82}.

The stability of the least-squares problem depends on the conditioning of $\bm W^{1/2}\bm L$, or equivalently of $\bm M$:
\begin{equation*}
    \kappa(\bm M)
    = \frac{\lambda_{\max}(\bm M)}{\lambda_{\min}(\bm M)}
    = \frac{\sigma_{\max}^2(\bm W^{1/2}\bm L)}{\sigma_{\min}^2(\bm W^{1/2}\bm L)},
\end{equation*}
where $\lambda_{\min/\max}$ and $\sigma_{\min/\max}$ denote the smallest and largest eigenvalues and singular values, respectively.
Because $\bm M=\sum_{i=1}^{n}w_i\bm\eta(x_i)\bm\eta(x_i)^\ast$ is the matrix from \eqref{eq:M}, points and weights satisfying an $L_2$ MZ inequality with distortion $\varepsilon$ obey
\begin{equation}\label{eq:singularvalues}
    1-\varepsilon
    \le \sigma_{\min}^2(\bm W^{1/2}\bm L)
    \le \sigma_{\max}^2(\bm W^{1/2}\bm L)
    \le 1+\varepsilon.
\end{equation}

Let $\widehat{\bm g}_\star$ be the exact least-squares coefficient vector and $\widehat{\bm g}_i$ be the $i$-th LSQR iterate.
The standard conjugate-gradient estimate gives
\begin{equation}\label{eq:iterationbound}
    \frac{\|\bm{\hat g}_\star-\bm{\hat g}_i\|_{\bm M}}{\|\bm{\hat g}_\star-\bm{\hat g}_0\|_{\bm M}}
    \le 2q(\bm M)^i
    \quad\text{with}\quad
    q(\bm M)
    \coloneqq
    \frac{\sqrt{\kappa(\bm M)}-1}
         {\sqrt{\kappa(\bm M)}+1},
\end{equation}
where $\|\bm z\|_{\bm M}^2=\bm z^\ast\bm M\bm z$, see e.g., \cite[Theorem~3.1.1]{Gre97}.
This estimate can be pessimistic: actual convergence also depends on the full spectrum of $\bm M$ and on the spectral distribution of the initial error.

A common rescaling of all weights leaves the least-squares minimizer and the condition number unchanged.
It is therefore natural to use the scale-invariant distortion
\begin{equation}\label{eq:scaleinvariantdistortion}
    \varepsilon_\star(\bm M)
    \coloneqq \inf_{c>0}\|c\bm M-\bm I_m\|_2
    =\frac{\kappa(\bm M)-1}{\kappa(\bm M)+1}.
\end{equation}
For a positive definite matrix, the infimum is attained by centering the extreme eigenvalues around one.
Consequently,
\begin{equation}\label{eq:qepsilon}
    q(\bm M)
    =\frac{\varepsilon_\star(\bm M)}
    {1+\sqrt{1-\varepsilon_\star(\bm M)^2}}.
\end{equation}
For $0\le q(\bm M)<1$ and a target relative error $0<\delta<1$, define the iteration count certified by \eqref{eq:iterationbound} as
\begin{equation}\label{eq:certifiediterations}
    n_{\rm iter}(\bm M)
    \coloneqq
    \min\big\{i\in\mathds N_0:2q(\bm M)^i\le\delta\big\}.
\end{equation}
For $0<q(\bm M)<1$, this equals
\begin{equation}\label{eq:certifiediterationsformula}
    n_{\rm iter}(\bm M)
    =\left\lceil\frac{\log(\delta/2)}{\log q(\bm M)}\right\rceil.
\end{equation}
The following result is a lower bound on the certified iteration count, rather than a claim that every LSQR run actually requires that many iterations.

\begin{proposition}\label{niterlower}
    Let $\mathcal F$ be either the function space associated with a complex ETF of $N$ vectors through \eqref{eq:framespace}, or the complete-graph space \eqref{eq:graphspace}; in both cases let $m=\dim\mathcal F$.
    Fix $m\le n\le N/2$ distinct points and positive weights for which the normal matrix $\bm M$ from \eqref{eq:M} is positive definite.
    Then
    \begin{equation}\label{eq:distortionlowerlsqr}
        \varepsilon_\star(\bm M)^2
        \ge \frac{N-n}{N+\beta n}
        \quad\text{with}\quad
        \beta=
        \begin{cases}
            \dfrac{N-m}{m-1} &\text{for the ETF space},\\
            m &\text{for the graph space}.
        \end{cases}
    \end{equation}
    Furthermore, for every $0<\delta<1$,
    \begin{equation}\label{eq:niterlower}
        n_{\rm iter}(\bm M)
        \ge
        \left\lceil
        \frac{2\log(\delta/2)}{\log(m/(24n))}
        \right\rceil.
    \end{equation}
\end{proposition}

\begin{proof}
    We choose the common rescaling of the weights that attains \eqref{eq:scaleinvariantdistortion}.
    The rescaled normal matrix then satisfies an MZ inequality with distortion $\varepsilon_\star(\bm M)$.
    Solving the ETF lower bound in Theorem~\ref{etfbound}, respectively the graph lower bound in Corollary~\ref{graphmzbound}, for the distortion gives \eqref{eq:distortionlowerlsqr}.

    In both cases $\beta m\le2N$: for ETFs this follows from $N\le m^2$, and for the graph space it follows from $N=m(m+1)/2$ and $\beta=m$.
    Since $n\ge m$,
    \begin{equation*}
        N+\beta n
        \le \Big(\frac Nm+\beta\Big)n
        =\frac Nm\Big(1+\frac{\beta m}{N}\Big)n
        \le\frac{3Nn}{m}.
    \end{equation*}
    Together with $n\le N/2$, this yields
    \begin{equation*}
        \varepsilon_\star(\bm M)^2
        \ge\frac{N-n}{N+\beta n}
        \ge\frac{m(N-n)}{3Nn}
        \ge\frac{m}{6n}.
    \end{equation*}
    By \eqref{eq:qepsilon},
    \begin{equation*}
        q(\bm M)
        \ge\frac{\varepsilon_\star(\bm M)}2
        \ge\sqrt{\frac{m}{24n}}.
    \end{equation*}
    Substitution into \eqref{eq:certifiediterationsformula} proves \eqref{eq:niterlower}.
\end{proof}

The restriction $n\le N/2$ is used only to obtain the simple constant in \eqref{eq:niterlower}.
This is a natural assumption, as the problem of function approximation becomes less relevant when the function is known on a majority of the domain.
The sharper distortion estimate \eqref{eq:distortionlowerlsqr} remains valid for every $m\le n<N$ and tends to zero as $n$ approaches $N$.

The constructive upper bound gives a converse statement for the iteration count.

\begin{lemma}\label{iterationupper}
    Let $\mathcal F$ be any $m$-dimensional function space and let $n>m$ be an integer.
    There are at most $n$ points and positive weights whose normal matrix $\bm M$ from \eqref{eq:M} satisfies, for every $0<\delta<1$,
    \begin{equation}\label{eq:niterupper}
        n_{\rm iter}(\bm M)
        \le
        \left\lceil
        \frac{2\log(\delta/2)}{\log(m/n)}
        \right\rceil.
    \end{equation}
\end{lemma}

\begin{proof}
    Set
    \begin{equation*}
        \varepsilon=\frac{2\sqrt{mn}}{m+n}.
    \end{equation*}
    Then $\sqrt{1-\varepsilon^2}=(n-m)/(n+m)$ and hence
    \begin{equation*}
        m\frac{1+\sqrt{1-\varepsilon^2}}
        {1-\sqrt{1-\varepsilon^2}}
        =n.
    \end{equation*}
    The constructive subsampling bound \eqref{eq:constructive-upper} therefore provides at most $n$ points and positive weights satisfying \eqref{eq:mz}.
    For the corresponding normal matrix, \eqref{eq:singularvalues} and \eqref{eq:qepsilon} give
    \begin{equation*}
        q(\bm M)
        \le\frac{\varepsilon}{1+\sqrt{1-\varepsilon^2}}
        =\sqrt{\frac mn}.
    \end{equation*}
    If $q(\bm M)=0$, then $n_{\rm iter}(\bm M)=1$ and \eqref{eq:niterupper} is immediate. Otherwise, substitution into \eqref{eq:certifiediterationsformula} proves \eqref{eq:niterupper}.
\end{proof}

One LSQR step requires multiplications by $\bm L$ and $\bm L^\ast$.
If $C_{\rm mv}(n,m)$ denotes their combined arithmetic cost, the overall arithmetic cost becomes
\begin{equation*}
    \operatorname{work}
    \approx C_{\rm mv}(n,m) n_{\rm iter}(\bm M),
\end{equation*}
apart from setup costs.
Depending on structure, $C_{\rm mv}(n,m)$ ranges from order $mn$ for a dense matrix, through order $n\log n$ for Fourier-type transforms, to order $n$ for sparse matrices with only a linear number of nonzero entries.
Proposition~\ref{niterlower} and Lemma~\ref{iterationupper} show that the certified iteration count improves only inverse-logarithmically with the oversampling ratio $n/m$.

Suppose for fixed $m$ that $C_{\rm mv}(n,m)$ is proportional to $n^\alpha$ with $\alpha>0$.
Writing $x=n/m>1$, the qualitative $x$-dependent factor in Proposition~\ref{niterlower} and Lemma~\ref{iterationupper} is
\begin{equation*}
    \frac{x^\alpha}{\log x},
\end{equation*}
which is minimized at $x=\exp(1/\alpha)$.
Thus a small oversampling factor is preferred.
Fourier-type costs change the optimizer slightly but lead to the same qualitative conclusion.
Function evaluations, data acquisition, and point construction add further $n$-dependent costs and may shift the optimum toward fewer samples.

Finally, producing a well-conditioned design can itself dominate the solve.
The implementation analyzed in \cite{BSU23} has a cubic dependence on $m$ in its initial frame-subsampling stage.
By contrast, independent Christoffel sampling gives order $m\log m/\varepsilon^2$ points with high probability for general spaces requiring merely a random draw \cite{CDL13}.
For structured trigonometric spaces, fast specialized rank-1 lattice constructions exist for order $m^2$ point constructions \cite{Kaemmerer20}.
Accordingly, the preferable sample size depends not only on the LSQR iteration estimate, but also on the structure of the evaluation matrix and on the cost of constructing and acquiring the samples.
 
\subsection*{Acknowledgments}

The author thanks Matthew~Fickus for drawing attention to harmonic ETFs arising from Singer difference sets.
The complete-graph construction was initially suggested during an interaction with an AI system.

\printbibliography

\end{document}